\documentclass[a4paper,reqno]{amsart}
\usepackage{enumerate}
\usepackage{pb-diagram}
\usepackage{microtype}
\usepackage{amsmath}
\usepackage{enumerate}
\usepackage{pb-diagram}
\usepackage{microtype}
\usepackage{amsmath}
\usepackage{amssymb, pb-diagram}
\usepackage[all]{xy}
\usepackage{graphicx} 
\usepackage{mathabx}
\usepackage{chapterbib}
\usepackage{epsfig}
\usepackage{amsfonts}
\usepackage{amssymb}
\usepackage{amsmath}   
\usepackage{amsthm}
\usepackage{color}
\usepackage{latexsym}
\usepackage{hyperref}
\newtheorem{thm}{Theorem}[section]

\newtheorem{prop}[thm]{Proposition}
\newtheorem{LM}[thm]{Lemma}
\newtheorem{cor}[thm]{Corollary}
 \theoremstyle{definition}
  \newtheorem{definition}{Definition}[section] 
    \newtheorem{conj}{Conjecture}[section] 
 \newtheorem{question}[thm]{Question}

    \newtheorem{rem}[thm]{Remark}
  \DeclareMathOperator {\Hom}{Hom}

   \DeclareMathOperator {\FF}{\mathfrak{F}} 
 \DeclareMathOperator {\PP}{\mathfrak{P}}

       \DeclareMathOperator {\ima}{im}

 \DeclareMathOperator {\dime}{dim} 
  
 \DeclareMathOperator {\asdim}{asdim}

  \DeclareMathOperator {\F}{\mathfrak{F}}

  \DeclareMathOperator {\Gcd}{Gcd}

 \DeclareMathOperator {\cd}{cd} 
 \DeclareMathOperator {\gd}{gd} 
  \DeclareMathOperator {\hd}{hd} 
  \DeclareMathOperator {\Rstt}{Rst}

\DeclareMathOperator {\h}{h} 
   \DeclareMathOperator {\vcd}{vcd} 
  \DeclareMathOperator {\pd}{pd} 
   \DeclareMathOperator {\Z}{\mathbb{Z}} 
     
    \DeclareMathOperator {\spli}{spli}

 \DeclareMathOperator {\D}{\Delta}

    \DeclareMathOperator {\CW}{CW}

 \DeclareMathOperator {\tos}{\twoheadrightarrow}

\numberwithin{equation}{section}

\begin{document}

\title[Cohomological invariants for groups]{Cohomological invariants and \\ the classifying space for proper actions}

\author{Giovanni Gandini}
\email{g.gandini@soton.ac.uk}

\address{School of Mathematics, University of Southampton, Southampton, SO17 1BJ UNITED KINGDOM}



\begin{abstract}
We investigate two open questions in a cohomology theory relative to the family of finite subgroups.  The problem of  whether the  $\FF$-cohomological dimension   is subadditive is reduced  to extensions by groups of prime order. 
We show that every finitely generated regular branch group has infinite rational cohomological dimension. Moreover, we prove that the first Grigorchuk group $\mathfrak{G}$  is  not contained in Kropholler's class ${\scriptstyle \mathbf H}\mathfrak F$.
\end{abstract}
\maketitle
\section{Introduction}
Let $G$ be a group and let $\FF$ be the class of finite groups. A $G$-$\CW$-complex is $proper$ if  all its cell stabilisers are in $\FF$. If a proper $G$-$\CW$-complex  $X$ has the property that for each $\FF$-subgroup $K$ of $G$  the fixed point subcomplex $X^{K}$ is contractible, then $X$ is called a \emph{classifying space for proper actions of $G$} (or a model for $E_{\FF}G$).  The equivariant $K$-homology of the classifying space for proper actions forms the left-hand side of the celebrated  Baum-Connes conjecture.  

 Every group $G$ admits a model for $E_{\FF} G$ by generalisations of the constructions of Milnor \cite{milnor} and Segal  \cite{segal}. The Bredon geometric dimension of $G$, denoted by $\gd_{\FF}G$, is the minimal dimension of a model for $E_{\FF}G$. The Bredon cohomological dimension  $\cd_{\FF}G$ plays a role analogous to that of  the integral cohomological dimension $\cd G$ in ordinary group cohomology and  is an algebraic counterpart of $\gd_{\FF} G$. In particular, $\cd_{\FF} G $ is finite if and only if $\gd_{\FF} G$ is finite \cite{luck-89}. However, both invariants are often very difficult to compute. Several possible ``easy'' geometric and algebraic invariants that guarantee their finiteness have been proposed by various authors \cite{guidosbook, abdeta, nucinkis-00}.

 On the geometric side Kropholler introduces the class of ${\scriptstyle \mathbf H}_{1}\mathfrak{F}$-groups  \cite{MR1246274}.  A group belongs to  ${\scriptstyle \mathbf H}_{1}\mathfrak{F}$   if there is a finite dimensional contractible $G$-$\CW$-complex $X$ with cell stabilisers in $\mathfrak{F}$.  
 The following has been open for almost 20 years.
  \begin{conj}[Kropholler, \cite{guidosbook}] Every  ${\scriptstyle \mathbf H}_{1}\mathfrak{F}$-group  $G$ admits a finite dimensional model for $E_{\FF}G$. 
 \end{conj}
A result proved independently by Bouc \cite{bouc} and Kropholler-Wall \cite{kropwall} implies that the augmented cellular chain complex $C_{*}(X)$ of any finite-dimensional contractible proper $G$-$\CW$-complex splits when restricted to the $\FF$-subgroups of $G$.  
Nucinkis introduces a cohomology theory  relative to a $G$-set $\D$  in order to algebraically mimic the behaviour of  ${\scriptstyle \mathbf H}_{1}\mathfrak{F}$-groups \cite{nucinkis-99}. This theory can be regarded as a cohomology relative to a class of proper short exact sequences as in  IX $\&$ XII  \cite{mac-lane-95} or as cohomology relative to the $\Z G$-module $\Z\D$. It generalises  cohomology relative to a subgroup to cohomology relative to   a family of subgroups. When dealing with the family of $\FF$-subgroups, we will refer to  this as $\FF$-cohomology. In this setup there is a well-defined $\FF$-cohomological dimension $\FF\cd G$  \cite{nucinkis-00}.
It is an open question whether every group of finite $\FF$-cohomological dimension lies in  ${\scriptstyle \mathbf H}_{1}\mathfrak{F}$. The converse holds by the result of Kropholler and Wall mentioned above.  Of course it is also unknown whether every group of finite $\FF$-cohomological dimension admits a finite dimensional model for  $E_{\FF}G$, and this  is conjectured in \cite{nucinkis-00}.  
It is well-known that for any group $\Gamma$,  $\cd_{\mathbb{Q}}\Gamma\leq\FF\cd \Gamma\leq\cd_{\FF}\Gamma\leq\gd_{\FF}\Gamma$, see for example \cite{bln}.

Remaining on the algebraic side it is important to mention that Bahlekeh, Dembegioti and Talelli conjecture in \cite{abdeta} that every group of finite  Gorenstein cohomological dimension $\Gcd G$ has finite Bredon geometric dimension. 
Most of these conjectures have positive answers in two important cases. The length $l(H)$ of an $\FF$-subgroup $H$ of $G$ is the  supremum over all $n$ for which
there is a chain $H_{0} < H_{1}<\ldots <H_{n}=H$. 
Firstly, by applications of a result of L\"uck  \cite{lueck-00} if $G$ has a bound on the lengths of its $\FF$-subgroups  then the finiteness of $\cd_{\FF}G$ is equivalent to the finiteness of $\Gcd G$ and  $\FF\cd G$.
 Secondly, if $G$  is a countable elementary amenable group then $\cd_{\mathbb{Q}}G<\infty$,  $ \FF\cd G<\infty$ and $\cd_{\FF}G<\infty$ are equivalent by a theorem of  Flores and Nucinkis \cite{flores-05}.\\ Let  $\dim$ be a function from the class of all groups to $ \mathbb{N} \cup \{\infty\}$. Then we say that $\dime$ is \emph{subadditive} if for every group extension $N\hookrightarrow  G\tos Q$ we have $\dime G \leq \dime N + \dime Q$.

The good behaviour of the $\FF$-cohomological dimension with respect to several group operations is known \cite{nucinkis-00} but its behaviour with respect to  taking  group extensions remains unclear. We investigate the subadditivity of the $\FF$-cohomological dimension and we prove that the $\FF$-cohomological dimension is subadditive if and only if the $\FF$-cohomological dimension is preserved under taking extensions by groups of prime order. Leary and Nucinkis \cite{leary-03} build a group extension such that $\cd_{\FF}N=\FF\cd N=2n$ and  $\cd_{\FF}Q=\FF\cd Q =0$ but $\cd_{\FF}G=3n$ and $\FF\cd G=2n$. Under extra conditions the behaviour of $\gd_{\FF} G$ under taking group extensions is known \cite{MR1851258, martinez-perez-02, MR2364823}. It is still unknown whether there exists  a group of infinite Bredon geometric dimension that is an extension of two groups of finite Bredon geometric dimension.

On the other hand it is known that the  Gorenstein cohomological dimension is subadditive \cite{abdeta}. The precise connections between the Gorenstein cohomological dimension  and the $\FF$-cohomological and rational cohomological dimensions are unclear. We show that $\Gcd G \leq \FF\cd G$, but it is unknown whether the finiteness of $\Gcd G$ implies the finiteness of $\FF\cd G$.  If there exists a group $G$ that has $\Gcd G <\infty$ or $\FF\cd G < \infty$ but admits no finite dimensional model for $E_{\FF}G$, then by the theorem of L\"uck $G$ can not have a bound on the lengths  of its $\FF$-subgroups. 

Branch groups are certain subgroups of the full automorphism groups of spherically homogeneous rooted trees.  Several examples of finitely generated periodic non-elementary amenable groups with no bound on the lengths of their $\FF$-subgroups lie in this class. Here we show that finitely generated regular branch groups have infinite rational cohomological dimension, which implies that the $\FF$-cohomological dimension and  the Bredon cohomological dimension are infinite as well. 

Let ${\scriptstyle \mathbf H}\mathfrak F$ be Kropholler's class  of \emph{hierarchically decomposable groups} \cite{MR1246274}. ${\scriptstyle \mathbf H}\mathfrak F$ is defined as the smallest class of groups containing the class $\FF$ and which contains a group $G$ whenever there is an admissible action of $G$ on a finite-dimensional contractible cell complex for which all isotropy groups already belong to ${\scriptstyle \mathbf H}\mathfrak F$. An important question in the area is to determine which branch groups belong to the class ${\scriptstyle \mathbf H}\mathfrak F$. 
Until the recent work \cite{team}, where groups with a strong global fixed point property are constructed,   the only way to show that  a group $G$ did not belong to ${\scriptstyle \mathbf H}\mathfrak F$ was to find a subgroup of $G$ isomorphic to the Thompson group $\bold{F}$.
Here we show that certain branch groups, such as  the first Grigorchuk group $\mathfrak{G}$ is not contained in the class ${\scriptstyle \mathbf H}\mathfrak F$. Furthermore, $\mathfrak{G}$  is a counterexample to a conjecture of Petrosyan \cite{petro} and answers in negative a question of Jo-Nucinkis \cite{nujo}.

\subsection*{Acknowledgements} The author would like to thank his supervisor Brita  E.A. Nucinkis and Laurent Bartholdi for their encouragement and advice.  
\section{Background}
Let  $\D$ be a $G$-set that satisfies  the following condition:
\begin{equation*}
\D^{H} \neq \emptyset \Longleftrightarrow   H\leq G,\,\,H \in \mathfrak{F}.\tag{$*$}
\end{equation*}
By Lemma 6.1 in \cite{nucinkis-99} any $G$-set satisfying the above generates the same cohomology theory. The standard example of such a $G$-set is the set of all cosets $Hg$ where $H$ is an $\FF$-subgroup and $g \in  G$. Note that  for our purpose  it is enough to consider one $H$ per conjugacy class, and if $G$ has maximal $\FF$-subgroups it is enough to consider these.

In $\FF$-cohomology exact sequences of $\Z G$-modules are replaced by $\FF$-split sequences, i.e. exact sequences of $\Z G$-modules that split under tensoring with $\Z\D$ over $\Z$. It turns out that a sequence of $\Z G$-modules has this property if and only if splits when restricted to each $\FF$-subgroup of $G$. A $\Z G$-module is \emph{$\FF$-projective} if  it is a direct summand of a module of the form $N \otimes \Z\D$ where $N$ is a $\Z G$-module. From the $\FF$-split surjection $\Z\D \tos \Z$ it follows that the category of $\Z G$-modules has enough $\FF$-projectives.  It is easy to show  that a relative version of the generalised Schanuel's lemma holds using Lemma 2.7 in \cite{nucinkis-99}. In particular, the $\FF$-projective dimension of any $\Z G$-module is well-defined. The $\FF$-cohomological dimension $\FF\cd G$ is defined as the $\FF$-projective dimension of the trivial $\Z G$-module $\Z$. 

For any $G$ there is a \emph{standard  $\FF$-projective resolution} of $\Z$  $$\bold{P}: \dots  \to P_{2}\stackrel{d_{2}}\to P_{1}\stackrel{d_{2}}\to P_{0}\stackrel{d_{0}} \tos\Z .$$
For $0 \leq i \leq n$ define $P_{i} = \mathbb{Z}(\D^{i})$ and the maps $d_{i} : P_{i+1} \to P_{i}$ are given by
$$d_{i}(\delta_{0}, \delta_{1}, \dots, \delta_{i})= \sum_{k=0}^{i}(-1)^{k}(\delta_{0}, \delta_{1}, \dots, \hat{\delta}_{k}, \dots , \delta_{i})$$ where $\hat{\delta}_{k}$ means  that $\delta_{k}$ is omitted  \cite{nucinkis-00}. Note that  the standard $\FF$-projective resolution $\bold{P}\tos \Z$ is the augmented cellular chain complex for a model of $E_{\FF} G$.  Of course $G$ has $\FF\cd G=0$ if and only if it is finite, and  $G$ has $\FF\cd G=1$ if and only if $G$ acts on a tree with finite stabilisers by Dunwoody's theorem \cite{Dcoh}.
For more details, the reader should consult \cite{nucinkis-99} and  \cite{nucinkis-00}.\

A $\Z G$-module $M$  admits a \emph{complete resolution} if there is an acyclic $\Z G$-projective complex $\bold{F} = \{(F_{i}, i \in \Z\}$, and a projective resolution $\bold{P} = \{P_{i}, i \in \mathbb{N}_{0}\}$ of $M$ such that F and P coincide in sufficiently high dimensions. The smallest dimension in which they coincide is called the coincidence index. A complete resolution of $M$ such that  $\Hom_{\Z G}(\bold{F}, Q)$ is acyclic  for every $\Z G$-projective module $Q$ is called  a complete resolution in the strong sense. $M$ is  \emph{Gorenstein projective} if it admits a complete resolution in the strong sense of coincidence index $0$. The category of $\Z G$-modules has enough  Gorenstein projectives and there is a well-defined notion of Gorenstein projective dimension. Now, as usual for a group $G$ the Gorenstein cohomological dimension $\Gcd G$  is defined as the Gorenstein projective dimension of the trivial $\Z G$-module $\Z$. A detailed study of the Gorenstein cohomological dimension can be found in \cite{cohodime, abdeta}. 
 \section{Group extensions and the $\mathfrak{F}$-cohomological dimension}
The class of groups of finite $\FF$-cohomological dimension is closed  under taking subgroups, HNN-extensions and  free products with amalgamation \cite{nucinkis-00}. Moreover in  \cite{nucinkis-00} it is shown that this class is closed under taking extension by groups of finite integral cohomological dimension.

We begin by recalling some results needed in the proof of the main proposition. 
\begin{LM}\cite[2.2]{nucinkis-00}\label{2.2}
Le $N\hookrightarrow   G \stackrel{\pi} \tos Q$ be a group extension and let $\mathfrak{H}$ be a family of
groups satisfying the following condition: if $H$ is a subgroup of $G$ and $H \in\mathfrak{H}$, then $\pi(H) \in \mathfrak{H}$. Then every $\mathfrak{H}$-split short exact sequence of $\Z Q$-modules is
$\mathfrak{H}$-split when regarded as a sequence of $\Z G$-modules.
\end{LM}
For a $\Z H$-module $M$,  we use the standard notation $M\uparrow_{H}^{G}:=\Z G \otimes_{\Z H} M$.
\begin{LM}\cite[8.2]{nucinkis-99}\label{Lemma8.2nuc99} Let $H$ be a subgroup of $G$ and let $A \hookrightarrow  B \tos C$ be an $\FF$-split short exact sequence  of kH-modules. Then the sequence $A\uparrow_{H}^{G} \hookrightarrow  B\uparrow_{H}^{G} \tos C\uparrow_{H}^{G}$ is an $\FF$-split sequence of kG-modules.
\end{LM}
Any $\Z G$-module $M$ induced up from an $\FF$-subgroup $H$ of $G$ is $\FF$-projective (Corollary 2.4, \cite{nucinkis-99}). This is not true for arbitrary subgroups $H$, but holds if $M$ is induced up from an $\FF$-projective $\Z H$-module. 
\begin{LM} \label{proind}
Let $H$ be a subgroup of $G$ and  $P$ be an $\FF$-projective $\mathbb{Z}H$-module. Then $P\uparrow_{H}^{G}$ is an $\FF$-projective $\mathbb{Z}G$-module.
\end{LM}
\begin{proof} If $\D=\bigsqcup_{\delta \in \Delta_{0}}  G / G_{\delta}$ is a $G$-set that satisfies condition $(*)$ then $\D$ has an $H$-orbit decomposition of the form $\bigsqcup_{\delta \in \Delta_{0}} (\bigsqcup_{g\in\Omega_{\delta}} H / H\cap G_{\delta}^{g})$, where $\Omega_{\delta}$ is a set of representatives of the double cosets $HgG_{\delta}$. Clearly $\D$ regarded as an $H$-set satisfies  condition $(*)$. Let $M$ be an $\FF$-projective $\Z H$-module, then by definition $M$ is a direct summand of $N \otimes  \Z \D$ for some $\Z H$-module $N$. Since induction is an exact functor, $M\uparrow_{H}^{G}$ is a direct summand of $   (N \otimes \Z \D)\uparrow_{H}^{G}$. The statement follows by the Frobenius Reciprocity $(N \otimes \Z \D)\uparrow_{H}^{G} \cong N \uparrow_{H}^{G} \otimes \Z \D$ (Exercise 2(a), 5, III \cite{brown-82}).
\end{proof}
\begin{LM} \label{permu}Suppose $G$ is a group of finite $\FF$-cohomological dimension equal to~ $n$. Then there is an $\FF$-projective resolution of $\Z$ of length $n$ consisting of permutation modules with finite stabilisers.
\begin{proof}
Since $\FF\cd G =n$, the general relative Schanuel's Lemma implies that the kernel $K_{n}$ of the standard $\FF$-projective resolution is $\FF$-projective and so $\Z(\D^{n}) \tos K_{n}$ splits, i.e. $K_{n}\oplus P \cong \Z(\D^{n})$. Let $\Z\hat{\D}$ be a module isomorphic to a  direct sum of countably many copies of $\Z(\D^{n})$. Then $K_{n}\oplus \Z \hat{\D}\cong \Z\hat{\D}$ and we have the required resolution:
\begin{equation*}\Z\hat{\D}\hookrightarrow  \Z(\D^{n-1})\oplus \Z\hat{\D}\to \dots\to\Z\D\tos\Z. \tag*{\qedhere}\end{equation*}
\end{proof}
\end{LM}
Note that in the proof above the relative Eilenberg swindle produces a permutation module; this does not hold for general $\FF$-projective modules. For further discussion consult Section 4,  \cite{nucinkis-00}.
\begin{cor}\label{gore} For any group $G$, $\Gcd G \leq \FF\cd G$.
\begin{proof} Every permutation $\Z G$-module with $\FF$-stabilisers is a Gorenstein projective $\Z G$-module by Lemma 2.21 \cite{cohodime}. The result now follows  from  Lemma \ref{permu}.
\end{proof}
\end{cor}
Martinez-P\'erez and Nucinkis prove using Mackey functors that for every virtually torsion-free group $G$ the equality $\vcd G = \FF\cd G$ holds  \cite{Martnu}. We give a proof of a weaker result, sufficient for our purpose, using an elementary method. 
\begin{prop}Let $G$ be torsion-free. Then $G$ has finite $\FF$-cohomological dimension equal to $n$ if and only if $G$ has finite cohomological dimension equal to~$n$.
\end{prop}
\begin{proof}If $\cd G =n$, then by Proposition 2.6 VIII in \cite{brown-82} there is a  $\mathbb{Z}G$-free resolution $F_{*}$ of $\mathbb{Z}$  of length $n$. 
Since $G$ is a torsion-free group, any $\mathbb{Z}G$-free module is $\FF$-projective and any acyclic  $\mathbb{Z}$-split $\mathbb{Z}G$-complex is $\FF$-split.  This shows that $F_{*}$ is  an $\FF$-projective resolution of $\mathbb{Z}$ of length $n$.\\
Now we consider the standard $\Z G$-free resolution of $\mathbb{Z}$:  $$\ldots\rightarrow F_{n-1}\rightarrow F_{n-2}\rightarrow \ldots \rightarrow F_{0} \twoheadrightarrow \mathbb{Z},$$ where $F_{i} =\mathbb{Z} (G^{i+1})$. By the above this is an $\FF$-split sequence. 
By the relative general Schanuel's lemma  applied to  $K_{n} \hookrightarrow  F_{n-1}\rightarrow F_{n-2}\rightarrow \ldots \rightarrow F_{0} \twoheadrightarrow \mathbb{Z}$ it follows that $K_{n}$ is $\FF$-projective.  In particular $K_{n}$, is a direct summand of $F_{n}$ and so it is $\Z G$-projective. \end{proof}
\begin{LM}[Dimension shifting]\label{dimshi} Let $N_{m}\hookrightarrow  N_{m-1}\to \dots \to N_{0} \tos L$ be an $\FF$-split exact sequence of $\Z G$-modules such that $\FF\pd N_{i}\leq n$ for all $0\leq i \leq m$. Then $\F\pd L \leq m+n$.
\begin{proof} We argue by induction on $m$. If $m=0$ then $N_{0}\cong L$ and $\F\pd L \leq n$. Let $k\geq 1$ and assume that the statement holds  for $m \leq k -1$. Consider the $\FF$-split short exact sequence $N_{k} \stackrel{\iota}{\hookrightarrow } N_{k-1}\tos \ima \iota$. By the induction hypothesis, $\F\pd (\ima \iota) \leq n+1$.  We have an $\FF$-split resolution of $L$ $\ima \iota \hookrightarrow  N_{k-2}\to \dots \to N_{0}\tos L$ of length $k-1$ made of modules of $\FF$-projective dimension at most $n$+1  and by the induction hypothesis we obtain $\F\pd L \leq (k-1)+(n+1) =k+n$.
\end{proof}
\end{LM}
\begin{prop}\label{withallf}Let $N \hookrightarrow  G \stackrel{\pi}\tos Q$ be a group extension with $\FF\cd Q\leq m$. Moreover, assume that any finite extension $H$ of $N$ has $\FF\cd H \leq n$. Then $\FF\cd G \leq n+m$.
\end{prop}
\begin{proof}For any finite extension $H$ of $N$, let
$$P_{n}\hookrightarrow  P_{n-1}\to \dots \to P_{0} \tos \mathbb{Z}$$ be a $\FF$-projective resolution of $\mathbb{Z}$ over $\mathbb{Z}H$. By   Lemma \ref{proind} and  Lemma \ref{Lemma8.2nuc99}, the resolution
$$P_{n}\uparrow_{H}^{G}\hookrightarrow  P_{n-1}\uparrow_{H}^{G}\to \dots \to P_{0}\uparrow_{H}^{G} \tos \mathbb{Z}\uparrow_{H}^{G}$$ is an $\FF$-projective resolution of  $\mathbb{Z}\uparrow_{H}^{G}$ over $\mathbb{Z}G$.\\
Now, Lemma \ref{permu} implies that there is an  $\FF$-projective resolution of $\mathbb{Z}$ over $\mathbb{Z}Q$ of the form:
$$\mathbb{Z}\D \cong K \hookrightarrow \mathbb{Z}\D_{m-1}\to \dots \to \mathbb{Z}\D_{0} \tos \mathbb{Z}.$$ By Lemma \ref{2.2} the sequence above is $\FF$-split when regarded as a $\mathbb{Z}G$-sequence.
Every permutation module $\mathbb{Z}\D_{i}$ and $\Z\D$ when regarded as a $\Z G$-module is isomorphic to some $\oplus_{j \in J} \mathbb{Z}\uparrow_{H_{j}}^{G}$ where $| H_{j} : N | < \infty$. To see this, consider the case of a homogeneous  $Q$-set $\Omega=Q/F$, and  regard $\Omega$ as  a $G$-set via $\pi$. Then $\Omega$ is isomorphic to $G/\pi^{-1}(F)$. If $|F|<\infty$ then $F\cong K/N$ where $[N : K]<\infty$ and $K \cong \pi^{-1}(F)$.    By the above $\FF\pd( \mathbb{Z}\uparrow_{H_{j}}^{G})<n$ and so the assertion follows by  Lemma \ref{dimshi}.
\end{proof}
\begin{cor}\label{directpro}If $G=H \times K$, where $\FF\cd H \leq n$ and $\FF\cd K \leq m$, then $\FF\cd G \leq n+m$. 
\begin{proof}
By Proposition \ref{withallf} we can assume $|K| < \infty$ and  we regard $G$ as an extension of $K$ by $H$. Any finite extension of $K$ by a finite subgroup of $H$ is finite and so it has $\FF$-cohomological dimension equal to $0$. The result now follows by Proposition~\ref{withallf}.
\end{proof}
\end{cor}
Proposition \ref{withallf} is the relative analogue of Corollary 5.2 in \cite{martinez-perez-02} but in the context of $\FF$-cohomology we are able to strengthen the result, as we shall see in Theorem \ref{fin}.

Since for virtually torsion-free groups the notion of $\FF$-cohomological dimension coincides with the notion of virtual cohomological dimension it is conceivable that taking finite extensions of groups of finite $\FF$-cohomological dimension does not raise the dimension. There are examples of non-virtually torsion-free groups  that are extensions of two virtually torsion-free groups of finite virtual cohomological dimension \cite{MR0473038}, but nonetheless these admit  finite dimensional  classifying spaces for proper actions \cite{bln}.

 In order to  reduce the extension problem to extensions by  groups of prime order we need the following observation.
\begin{LM}\label{p} Let $\mathfrak{P}$ be the class of $p$-groups. When considering the standard $\FF$-projective resolution $P_{*} \tos Z$ we can replace the $G$-set  $\D$ by $\D_{\mathfrak{P}}$, where  $\D_{\mathfrak{P}}= \bigsqcup _{P \leq G, \,P \in \mathfrak{P}\cap \mathfrak{F}}G/P$.
\begin{proof}The result is an immediate consequence of  i) and ii) of Proposition 2.14 \cite{nuclepre}.
\end{proof}
\end{LM}
\begin{rem}In view of Chouinard's Theorem \cite{chouinard} it is natural to ask if $\FF$-cohomology can be reduced  to a cohomological theory relative to the  family $\mathfrak{E}$ of finite elementary abelian subgroups. This is not the case; to see this let $P$ be a non-elementary abelian finite $\PP$-group  and let $\{H_{i}\}$ be the family of conjugacy classes of its elementary abelian subgroups. The short exact sequence $K \hookrightarrow  \oplus_{i} \Z P/H_{i} \stackrel{\pi}{\tos} \Z$ is $\mathfrak{E}$-split but does not split over $\Z P$.
 Let  $\sigma_{i} \in \Z [P/H_{i}]$ denote the sum of the cosets of $H_{i}$ in $P$, $\sigma_{i}=\sum_{pH_{i} \in P/H_{i}}pH_{i}$. Since $P/H_{i}$ is a transitive $P$-set, the only well-defined $\Z P$-map from $\Z$ to $\Z P/H_{i}$ is the map $1 \mapsto m_{i} \sigma_{i}$ where $m_{i}$ is a non-zero integer. Any $\Z P$-map $\iota : \Z \to\oplus_{i} \Z P/H_{i}\cong \Z P/H_{1}\oplus \Z P/H_{2}\oplus \dots \oplus \Z P/H_{n}$ is defined by $1\mapsto (m_{1}\sigma_{1},\dots,m_{n}\sigma_{n} )$ for some choice of $\{m_{1}, m_{2},\ldots ,m_{n}\}$.  Since $ \pi\circ \iota(1)=\sum_{i=1}^{n}m_{i}[P : H_{i}]=\sum_{i=1}^{n}m_{i}p^{n_{i}}\neq 1$ ($n_{i}\neq 0$ for all $i$), $\pi$ does not split over $\Z P$. 
\end{rem}

\begin{thm}\label{fin} Let $N \hookrightarrow  G \tos Q$ be a group extension. Moreover, assume that for every subgroup $H$ of $G$ with $\F\cd H \leq n$, every  extension $L$ of $H$ by a group of prime order has $\FF\cd L \leq n $. Then $\FF\cd G \leq \F\cd H+\F\cd Q$.
\begin{proof}
Arguing as in  Proposition \ref{withallf} the problem can be reduced to extensions by groups of prime power order using  Lemma \ref{p}. Now, if $N \hookrightarrow  G \tos P$ is such an extension then the quotient group $P$ is nilpotent of prime power order and so for any $p^{n}$ dividing  $|P|$ there exists a normal subgroup $S$ of $G$,  $N\leq S\leq G$ such that $S/N$ has order $p^{n}$ and the result is obvious.
\end{proof}
\end{thm}
Note first that if $N \hookrightarrow  G \tos Q$ is a group extension such that $\gd_{\FF}N=n$ and $|Q|=k$ then $\gd_{\FF}G \leq nk$ \cite{MR1851258}.
It is unknown if the finiteness of the $\FF$-cohomological dimension is preserved under taking (finite) extensions.  However this is the case for countable  elementary amenable groups.
\begin{prop}Let $N \hookrightarrow  G \tos Q$ be a group extension with $\FF\cd N\leq n$, $\FF\cd Q\leq m$ and such that  $G$ is countable elementary amenable. Then $\FF\cd G \leq n+m+1$.
\end{prop}
\begin{proof} The rational cohomological dimension of a group $G$ $\cd_{\mathbb{Q}}$ is defined as the $\mathbb{Q}G$-projective dimension of the trivial $\mathbb{Q}G$-module $\mathbb{Q}$. If the trivial $\Z G$-module $\Z$ admits a resolution of length $n$  made of permutation modules with $\FF$-stabilisers then tensoring it with $\mathbb{Q}$ over $\Z$ we obtain  that $\cd_{\mathbb{Q}}G\leq n$. In particular for every group $\cd_{\mathbb{Q}}G\leq \FF\cd G $. Corollary 3.3 \cite{nucinkis-00}  implies that for any group $G$  $ \FF\cd G \leq \cd_{\mathfrak{F}} G$. Let $\h G$ be the Hirsch length of an elementary amenable group $G$.  The inequality $\h G \leq\cd_{\mathbb{Q}}G$ holds by Lemma 2 in \cite{hillman-91}. Let $\hd_{R}G$ denote the homological dimension of $G$ over $R G$. If $G$ is any countable group $G$ and $R$ is a commutative ring of coefficients, then the following are well known \cite{bieri-81, nucinkis-04}:
$$\hd_{R} G \leq \cd_{R} G \leq \hd_{R} G +1,$$ 
$$\hd_{\mathfrak{F}} G \leq \cd_{\mathfrak{F}} G \leq \hd_{\mathfrak{F}} G +1.$$ 
The class of elementary amenable groups is subgroup-closed and quotient-closed. By Theorem 1 in \cite{hillman-91}  $\h G = \h N +\h Q$,  and an immediate application of Theorem 1 in \cite{flores-05} gives the result. 
\end{proof}
Furthermore, Serre's construction included in 5.2 V of \cite{DD} shows that, given   a finite extension $ N \hookrightarrow  G \tos Q$ with $\FF\cd N =n$ and  $|Q|=k$, there exists an exact $\Z G$-resolution of $\Z$ made of permutation modules with  stabilisers in $\FF$ of length $nk$. However, it is unclear if this resolution is $\FF$-split and  this suggest a more general question.
\begin{question}\label{ques1} Suppose $G$ is a group that admits a resolution of finite length of the trivial $\Z G$-module $\Z$ made of permutation modules with stabilisers in $\FF$. Does $G$ have finite $\FF$-cohomological dimension?
\end{question}
\begin{rem}
Arguing as in Corollary \ref{gore}, every group admitting a resolution as in the question above has finite Gorenstein cohomological dimension. It is unknown if the converse holds.  
\end{rem}
\section{Branch groups, rational cohomological dimension and ${\scriptstyle \mathbf H}\mathfrak F$}

As mentioned in the introduction by the result of L\"uck  \cite{lueck-00}, every group of finite $\FF$-cohomological dimension which has a bound on the lengths of its $\FF$-subgroups admits a finite dimensional classifying space for proper actions.\
 
 In this section we calculate the rational cohomological dimension of some finitely generated periodic groups with no such bound.  Moreover, we look into the problem of determining which branch groups lie in the class ${\scriptstyle \mathbf H}\mathfrak F$. We give a purely algebraic criterion, from which it follows that the first Grigorchuk group $\mathfrak{G}$ is not contained in ${\scriptstyle \mathbf H}\mathfrak F$.
 
 Usually if one wants to prove that a group $G$ has finite $\cd_{\mathbb{Q}}G$ either one finds a suitable finite dimensional $G$-space or decomposes the group $G$ in  order to control its rational cohomological dimension. On the other hand one usually proves that $G$ has infinite $\cd_{\mathbb{Q}}G$ in the following way. Since having finite $\cd_{\mathbb{Q}}G$ is a subgroup-closed property it is enough to find an infinite  chain of subgroups of  strictly increasing rational cohomological dimension. For the groups we consider in this section there is no such chain, although  we are able to establish their dimension because there is a chain of groups of strictly increasing cohomological dimension that uniformly embeds  in them.

A group $G$ is  $R$-torsion-free if the order of every finite subgroup of $G$ is invertible in the ring $R$. 
 \begin{thm}\cite[V 5.3]{DD}\label{DD} Let $G$ be a group and let $H$ be a subgroup of $G$ of finite index. If $G$ is  $R$-torsion-free, then $\cd_{R} H= \cd_{R}G$.
\end{thm}
\begin{definition}\cite[1.1]{sauer} Let $H$ and $K$ be countable groups:
A map $\phi : H \to K $ is called a \emph{uniform embedding} if for every sequence of pairs
$(\alpha_{i}, \beta_{i}) \in H \times H$ one has:
$$\alpha_{i}^{-1}\beta_{i}\to \infty \mbox{ in } H \iff \phi(\alpha_{i})^{-1}\phi(\beta_{i})\to \infty \mbox{ in } K.$$
Where $\to \infty$ means eventually leaving every finite subset.
\end{definition}
Note that this embedding is not necessarily a group homomorphism. Sauer proved the following remarkable result.
\begin{thm}\cite[1.2]{sauer}\label{shalom} Let $G$ and $H$ be countable groups and let $R$ be a commutative ring. If  $\cd_{R}H <\infty$ and $H$ uniformly embeds  in $ G$, then  $\cd_{R}H \leq \cd_{R}G$.
\end{thm}
Two groups $H$ and $G$ are said to be \emph{commensurable} if there exist $H_{1}\leq H$,  $G_{1}\leq G$ such that  $[H : H_{1}]< \infty$, $[G : G_{1}]< \infty$ and $H_{1}\cong G_{1}$. A group $G$ is \emph{multilateral} if it is infinite and  commensurable to some proper direct power of itself.

\begin{thm}\label{regu} Let $G$ be a finitely generated  multilateral group.
Then $\cd_{\mathbb{Q}}G=\infty$.
\begin{proof} If $A$ and $B$ are two commensurable groups then by Theorem \ref{DD} it follows that $\cd_{\mathbb{Q}}A=\cd_{\mathbb{Q}}B$.  Let $G$ be a finitely generated infinite group commensurable with $G^{k}$ for some $k>1$. First we show that $G$ is commensurable to $G^{k^{n}}$ for any $n\geq 1$. We proceed by induction on $n$. The base case  $n=1$ is obvious. Now $G^{k^{n+1}}\cong (G^{k^{n}})^{k}$,  by the induction hypothesis $G$ is commensurable to $G^{k^{n}}$ and so $G^{k}$ is commensurable to $(G^{k^{n}})^{k}$. Since $G$ is commensurable to $G^{k}$ and  commensurability is transitive, we obtain that $G$ is  commensurable to $G^{k^{n+1}}$. 
By Exercise IV.A.12 \cite{harpe} there is an isometric embedding $\mathbb{Z}  \hookrightarrow G$, from which it follows that there is an isometric embedding $\mathbb{Z}^{k^{n}}  \hookrightarrow G^{k^{n}}$. An application of Theorem \ref{shalom} gives $k^{n}=\cd_{\mathbb{Q}} \mathbb{Z}^{k^{n}} \leq \cd_{\mathbb{Q}} G^{k^{n}}=\cd_{\mathbb{Q}}G$. Since the last inequality holds for every non-negative integer $n$ we have $\cd_{\mathbb{Q}}G=\infty$.   
\end{proof}
\end{thm}

The converse of the theorem above does not hold. In fact the finitely generated ${\scriptstyle \mathbf H}_{2}\FF$-group of infinite cohomological dimension $\Z \wr\Z$ is not commensurable to any of its proper direct powers. \\
Tyrer Jones in \cite{tyrer} constructs a finitely generated non-trivial group $G$ isomorphic to its own square; as an immediate application of Theorem \ref{regu} we obtain that $\cd_{\mathbb{Q}}G=\infty$. 
\begin{rem}If $G$ is a finitely generated multilateral group, then  the proof of  Theorem 3 \cite{smith} extends verbatim  by replacing $G^{n}$ with $G^{k^{n}}$ to conclude that $\asdim G=\infty$.
For many groups the finiteness of the asymptotic dimension agrees with the finiteness of the rational cohomological dimension, although Sapir in \cite{sapir} constructed a $4$-dimensional closed aspherical manifold M such that the fundamental group $\pi_{1}(M)$ coarsely contains an expander, and so $\pi_{1}(M)$  has infinite asymptotic dimension but finite cohomological dimension.
\end{rem}

Note that if $G$ is a finitely generated infinite group such that $G^{n} \hookrightarrow G$ with $n>1$, then arguing as in Theorem \ref{regu} we obtain that $G$ has infinite rational cohomological dimension.  Of course if $G$ is not periodic  this shows that it contains a free abelian group of infinite countable rank. For example, it is well-known that  for  Thompson's group $\bold{F}$ we have the embedding  $\bold{F} \times \bold{F}\hookrightarrow \bold{F}$. 
\begin{cor}\label{reg} Every finitely generated regular branch group has infinite rational cohomological dimension.
\begin{proof}
For the precise definition of a regular branch group the reader is referred to \cite{branch}.  
Let $\mathcal{T}$ be an $m$-ary regular rooted tree and $G$ a finitely generated regular branch group acting on $\mathcal{T}$. By definition if $G$ is branching over $K$ then $[G : K]<\infty$ and $[\psi(K) : K^{m}]< \infty$, where $\psi$ is the embedding of the stabiliser of the first level in the direct product $G^{m}$. Since $\psi(K) \cong K$ we have that $K$ is commensurable with $K^{m}$. $K$ is finitely generated and so an application of Theorem \ref{regu} gives $\cd_{\mathbb{Q}}K = \infty$.  The finiteness of the rational cohomological dimension is preserved under taking subgroups  and so  we have  $\cd_{\mathbb{Q}}G=\infty$.
\end{proof} 
\end{cor}
Since the Gupta-Sidki group $\overline{\overline{\Gamma}}$ is a finitely generated regular  branch group  \cite{branch} we obtain as an application of Corollary \ref{reg} that  $\cd_{\mathbb{Q}}\overline{\overline{\Gamma}}=\infty$. Note that since  $\overline{\overline{\Gamma}}$ is a $p$-group with no bound on the orders of its elements it has no bound on the lengths of its $\FF$-subgroups.

\begin{rem} 
We have proved Corollary \ref{reg} in the context of regular branch groups for convenience only. In fact it was pointed out to the author by Laurent Bartholdi,  that also holds for the more general branch groups defined as follows.\\
A group $G$ is \emph{branch} if   it admits a \emph{branch
structure}: there exist a sequence of groups $\{G_{i}\}_{i\in \mathbb{N}}$, a sequence of positive integers $\{n_{i}\}_{i\in \mathbb{N}}$ and a sequence of homomorphisms $\{\phi_i\}_{i\in \mathbb{N}}$ such that $G\cong G_{0}$, and for each $i$,  \begin{enumerate}
\item
$\phi_i: G_i\to G_{i+1}\wr\Sigma_{n_i}$ has finite kernel and finite cokernel,
\item
 the image of each $\phi_i $ acts transitively on $\Sigma_{n_i}$, and the stabiliser of any $j\in{1,\dots,n_i} $ maps onto $G_{i+1}$.
\end{enumerate}

The structure is non-trivial if all $n_i \ge 2$, and the $\phi_i$ are injective.
It is easy to see that a branch group as above is a branch group  in the geometric sense of  \cite{branch}.
Now, let $G$ be a finitely generated infinite group that admits a sequence of groups $\{G_{i}\}_{i\in \mathbb{N}}$  and a sequence of integers $\{n_{i}\}_{i\in \mathbb{N}}$,   such that  $G\cong G_{0}$ and for each $i$, $G_{i}$ is commensurable with $G_{i+1}^{n_{i}}$. Arguing as in Theorem \ref{regu} it is easy to see that, if all  $n_{i} \geq 2$, the rational cohomological dimension of $G$ is infinite. Arguing as in Corollary \ref{reg} we deduce that every finitely generated branch group has infinite rational cohomological dimension.
\end{rem}
A group $G$ is said to have \emph{jump cohomology of height n over} $R$ if there exists an integer $n\geq 0$ such that any subgroup $H$ of finite cohomological dimension over $R$ has $\cd_{R}(H)\leq n$.
\begin{thm}\cite[3.2]{petro}\label{petro} Let $G$ be an $R$-torsion-free ${\scriptstyle \mathbf H}\mathfrak F$-group  with jump cohomology of height $n$ over $R$. Then, $\cd_{R}G \leq n$. In particular, any ${\scriptstyle \mathbf H}\mathfrak F$-group $G$ has jump rational cohomology of height $n$ if and only if $\cd_{\mathbb{Q}} G \leq n$.
\end{thm}
\begin{LM}\label{counta}
Let $G$ be a countable  group with $ \cd_{\mathbb{Q}} G < \infty$. Then there exists a finitely generated subgroup $H$ of $G$  such that $$\cd_{\mathbb{Q}}H\leq \cd_{\mathbb{Q}}G \leq \cd_{\mathbb{Q}} H+1.$$
Moreover, if  $\FF\cd G< \infty$ then there  exists a finitely generated subgroup $K$ such that 
$$\FF\cd K\leq \FF\cd G \leq \FF\cd K+1. $$

\end{LM}
\begin{proof} The statement for the rational cohomological dimension follows from Theorem 4.3 in \cite{bieri-81} and for the $\FF$-cohomological dimension it follows from Proposition 2.5 in \cite{nucinkis-00}.
\end{proof}

 We say that a group $G$ is   \emph{strongly multilateral} if it is multilateral and every finitely generated subgroup of $G$ is  commensurable  to some direct power of $G$.
\begin{thm}\label{mainn}Every finitely generated  strongly multilateral group    has jump rational cohomology
of height $1$. 
\begin{proof} Let $G$ be a  finitely generated strongly multilateral group. Then $G$ by Theorem \ref{regu} has infinite rational cohomological dimension. 
Suppose $H$ is a  finitely generated infinite subgroup of $G$, then by hypothesis $H$ is commensurable with some direct power of $G$ and so by Theorem \ref{DD} $\cd_{\mathbb{Q}}H=\infty$. 
 Suppose now that $H$ is an  infinitely generated subgroup of $G$ of finite  rational cohomological dimension. By Lemma \ref{counta} there exists $K\leq H$ such that $K$ is finitely generated and  $\cd_{\mathbb{Q}}K\leq \cd_{\mathbb{Q}}H \leq \cd_{\mathbb{Q}} K+1$. By the above $K$ can not be infinite and so $\cd_{\mathbb{Q}}H=1$.  
 \end{proof}
\end{thm}
\begin{cor}\label{4.8}
If $G$ is a finitely generated  strongly multilateral group, then $G$ is not in ${\scriptstyle \mathbf H}\mathfrak F$.
\begin{proof}The group $G$ has jump rational cohomology of height $1$ but infinite rational cohomological dimension and so by Theorem \ref{petro} $G \notin {\scriptstyle \mathbf H}\mathfrak F$.
\end{proof}
\end{cor}
The first Grigorchuk group $\mathfrak{G}$ is an infinite periodic finitely generated amenable group  \cite{grigo80}.  $\mathfrak{G}$ can be obtained as a subgroup of the automorphism group of the rooted binary tree. Since $\mathfrak{G}$ has infinite locally finite subgroups \cite{lfgri}, it has no bound on the lengths of its $\FF$-subgroups. For the definition and further details the reader should consult \cite{branch} or  \cite{harpe}. 
\begin{thm}\label{main} The  first Grigorchuk group $\mathfrak{G}$   has jump rational cohomology of height $1$, and has infinite rational cohomological dimension. Hence $\mathfrak{G}$ is not in ${\scriptstyle \mathbf H}\mathfrak F$.
\begin{proof} By VIII.14 and .15 \cite{harpe} $\mathfrak{G}$ is commensurable with its square,  infinite and finitely generated.  Any finitely generated infinite subgroup of $\mathfrak{G}$ it is commensurable with $\mathfrak{G}$ \cite{G-W}  and so by Corollary \ref{4.8}  $\mathfrak{G}\notin{\scriptstyle \mathbf H}\mathfrak F$.
\end{proof}
\end{thm}
\begin{rem} Theorem \ref{main} has two more consequences.\\ 
\emph{Conjecture} \cite{petro} For every group $G$ without $R$-torsion the following are equivalent.
\begin{itemize}
\item $G$ has jump cohomology of height $n$ over $R$.
\item $G$ has periodic cohomology over $R$ starting in dimension $n+1$.
\item $\cd_{R}G \leq n$.
\end{itemize}
Obviously from Theorem \ref{main} it follows that $\mathfrak{G}$ is a counterexample to the above conjecture.\

  Jo-Nucinkis in \cite{nujo} ask the following.\\  
  \emph{Question.} Let $G$ be a group such that every proper subgroup $H$ of $G$ of finite Bredon cohomological dimension satisfies $\cd_{\FF}H \leq n$ for some positive integer $n$. Is $\cd_{\FF}G <\infty$?\
  
Since a group $G$ has rational cohomological dimension equal to 1 if and only if it has Bredon cohomological dimension equal to 1 \cite{Dcoh},  Theorem \ref{main} shows that $\mathfrak{G}$  provides a negative answer to their question. \\
Given Theorem \ref{main}, it is easy to see that for any $n\geq 1$, the group $\mathfrak{G}\times \Z^{n-1}$ has infinite rational
cohomological dimension and jump rational cohomology of height $n$.
\end{rem}

The question of Jo-Nucinkis is a  ``proper actions version'' of an older question of Mislin-Talelli that asks if there exists a torsion-free group with jump integral cohomology but infinite cohomological dimension.
Note that every virtually torsion-free branch group $G$  contains a free abelian group of infinite countable rank. To see this take a ray and all edges just hanging off it. Than there is a non-trivial element of infinite order $a_{n}$ hanging off each edge since the rigid stabiliser of the $n$th-level $\Rstt_{G}(n)$ has finite  index in $G$ and $G$ is spherically transitive. These elements generate distinct  infinite cyclic subgroups of $G$ that obviously commute since they act on distinct subtrees and so they generate $\bigoplus_{\mathbb{N}} \Z$.
 This implies that $G$ has infinite rational cohomological dimension and does not have jump rational cohomology. Moreover, no torsion-free subgroup of finite index in $G$ can answer Mislin-Talelli question. 
A  more detailed study of the subgroup lattices of  virtually torsion-free branch groups would be very interesting. In fact it is unknown if there exists  a torsion-free group  $G\in{\scriptstyle \mathbf H}\FF\backslash {\scriptstyle \mathbf H}_{3}\FF$.  
\begin{question}Does every  finitely generated periodic  regular branch group have a finitely generated  strongly multilateral subgroup?
\end{question}

\begin{rem} Note that if $G$ is an ${\scriptstyle \mathbf H}\mathfrak F$-group, then $\Gcd G<\infty$ implies that $\cd_{\mathbb{Q}}G <\infty$. This can be shown in the following way. First we recall that  $\spli (RG)$ is the supremum of the projective  lengths of the injective $RG$-modules.  $\kappa(RG)$ is the supremum of the projective dimensions of the $RG$-modules that have finite $RF$-projective dimension for all $\FF$-subgroups of $G$. For any group  $G$, $\Gcd G <\infty$ if and only if $\spli(\mathbb{Z}G)<\infty$ by Remark 2.10 in \cite{cohodime}. Assume now that $G$ is an ${\scriptstyle \mathbf H}\mathfrak F$-group of finite Gorenstein  cohomological dimension. By Theorem C in \cite{cornick-98} $\spli(\mathbb{Q}G)=\kappa(\mathbb{Q}G)$. By \cite{geldgru} $\spli(\mathbb{Q}G)\leq\spli(\mathbb{Z}G)$; in particular  if $\spli(\mathbb{Q}G)<\infty$  then $\kappa(\mathbb{Q}G)< \infty$. Since $\mathbb{Q}$ is $\mathbb{Q}F$-projective for every $\FF$-subgroup $F$ of $G$ we have $\cd_{\mathbb{Q}}G <\infty$.

It is  known from recent work of Dembegioti and Talelli \cite{tade10} that the notions of a Gorenstein projective module and a cofibrant module coincide over ${\scriptstyle \mathbf H}\mathfrak F$-groups.  We suspect that the Gorenstein projective modules over an  ${\scriptstyle \mathbf H}\mathfrak F$-group $G$ are exactly direct summands of $\Z G$-modules obtained as extensions of permutation modules with $\FF$-stabilisers. If this holds then the inequality $\cd_{\mathbb{Q}}G\leq\Gcd G$ would be immediate.
 
It would be  interesting to compute the Gorenstein cohomological dimension of $\mathfrak{G}$. In fact,  $\mathfrak{G}$ could be a  counterexample to the conjecture of Bahlekeh, Dembegioti and Talelli.
\end{rem}
A group $G$ is \emph{just infinite} if it is infinite and every non-trivial normal subgroup of $G$ has finite index. It is well known that the group $\mathfrak{G}$ has this property \cite{branch}.

\begin{cor} $\mathfrak{G}$ does not contain a group of finite $\FF$-cohomological dimension for which the extension property fails to be subadditive.  \end{cor}
\begin{proof} $\mathfrak{G}$ is just infinite and by Theorem \ref{DD} every normal subgroup $N$ of $\mathfrak{G}$  has infinite rational cohomological dimension,  so $\FF\cd N =\infty$.  Assume  $L$ is a subgroup of $\mathfrak{G}$ such that $H \hookrightarrow  L \tos Q$, with $\FF\cd H=n$, $|Q|<\infty$ and $n<\FF\cd L <\infty$. Then, by Theorem \ref{main} it follows that $\mathfrak{G}$ has jump rational cohomology of height $1$ and $L$ is not finitely generated. 
From Lemma \ref{counta} $\cd_{\mathbb{Q}}L\leq 1$. By Dunwoody's theorem \cite{Dcoh} $\cd_{\mathbb{Q}}L\leq 1$ if and only if $L$ acts on a tree $T$ with finite stabilisers. We can assume $|L|=\infty$ and the tree $T$  is a one dimensional model for $E_{\FF}L$, so $\cd_{\mathbb{Q}}L=\FF\cd L=\cd_{\FF}L=\gd_{\FF}L=1$ and the result follows from Theorem \ref{DD}.
\end{proof}

\bibliographystyle{alpha}
\bibliography{math}
\end{document}